\documentclass[a4paper,11pt]{article}
\usepackage{amsfonts,amsmath,amssymb,amsthm}
\usepackage[applemac]{inputenc}
\usepackage[T1]{fontenc}
\usepackage[english]{babel}
\usepackage[all]{xy}
\usepackage{verbatim}
\usepackage{hyperref}
\hypersetup{colorlinks,%
citecolor=red,%
filecolor=black,%
linkcolor=red,%
urlcolor=blue,%
pdftex}

\newcommand{\Q}{\mathbb{Q}}
\newcommand{\C}{\mathbb{C}}
\newcommand{\Z}{\mathbb{Z}}

\newcommand{\h}{\mathrm{H}} 

\newcommand{\iHom}{{\mathcal{H}om}}

\renewcommand{\O}{\mathcal{O}}

\renewcommand{\k}{\textit{\textbf{k}}}

\DeclareMathOperator{\Hom}{Hom}

\DeclareMathOperator{\id}{id}
\DeclareMathOperator{\Spec}{Spec}

\DeclareMathOperator{\Pic}{Pic}

\DeclareMathOperator{\gr}{gr}
\DeclareMathOperator{\Lie}{Lie}

\DeclareMathOperator{\Ext}{Ext}
\DeclareMathOperator{\Ker}{Ker}
\DeclareMathOperator{\Coker}{Coker}

\newcommand{\x}{\times}
\newcommand{\ox}{\otimes}
\newcommand{\iso}{\cong}
\newcommand{\et}{{\rm et}}

\newcommand{\inv}{^{-1}}
\newcommand{\dual}{^{\vee}}
\newcommand{\dfni}[1]{\emph{#1}\index{#1}}

\newtheorem{thr}{Theorem}[section]
\newtheorem{lmm}[thr]{Lemma}
\newtheorem{prp}[thr]{Proposition}
\newtheorem{crl}[thr]{Corollary}
\theoremstyle{definition}\newtheorem{dfn}[thr]{Definition}
\theoremstyle{remark}\newtheorem{rmk}[thr]{Remark}
\theoremstyle{remark}\newtheorem{exm}[thr]{Example}
\title{Extensions of Formal Hodge Structures}
\author{Nicola Mazzari}
\makeindex
\begin{document}
\maketitle
\begin{abstract}
	We define and study the properties of the category ${\sf FHS}_n$ of formal Hodge structure of level $\le n$ following the ideas of L. Barbieri-Viale who discussed the case of level $\le 1$. As an application we describe the generalized Albanese variety of Esnault, Srinivas and Viehweg via the group $\Ext^1$ in ${\sf FHS}_n$. This formula generalizes the classical one to the case of proper but non necessarily smooth complex varieties. 
\end{abstract}
	\section*{Introduction}
	The aim of this work is to develop the program proposed by S. Bloch, L. Barbieri-Viale and V. Srinivas (\cite{bloch.srinivas:ehs},\cite{bv:fht}) of generalizing Deligne mixed Hodge structures providing a new cohomology theory for 
	complex algebraic varieties. In other words to construct and study cohomological invariants of (proper) algebraic schemes over $\C$ which are finer than the associated mixed Hodge structures in the case of 
	singular spaces.
	For any natural number $n>0$ (the level) we construct an abelian category, ${\sf FHS}_n$, and a family of functors 
	\[
		\h^{n,k}_\sharp:({\sf Sch/\C})^\circ\rightarrow {\sf FHS}_n\qquad 1\le k\le n
	\]
	such that 
	\begin{enumerate}
		\item The category ${\sf MHS}_n$ of mixed Hodge structure of level $\le n$ is a full sub-category of ${\sf FHS}_n$.
		\item There is a forgetful functor $f:{\sf FHS}_n\to {\sf MHS}_n$ s.t. $f(\h^{n,k}_\sharp(X))=H^n(X)$ (functorially in $X$) is the usual mixed Hodge structure on the Betti cohomology of $X$, i.e. $\h^n(X):=\h^n(X_{\rm an}, \Z)$.
	\end{enumerate}

	Roughly speaking the sharp cohomology objects $\h^{n,k}_\sharp(X)$ consist of the singular cohomology groups $\h^{n}(X_{\rm an},\Z)$, with their mixed Hodge structure, plus some extra structure. We  remark that $\h^{n,k}_\sharp(X)$ is completely determined by the mixed Hodge structure on $\h^n(X)$ when $X$ is proper and smooth;  the extra structure shows up only when $X$ is not proper or singular.

	The motivating example is the following. Let $X$ be a proper algebraic scheme over $\C$. Denote $\h^i(X):=\h^i(X_{\rm an},\Z)$, $\h^i(X)_\C:=\h^i(X)\ox \C$ and let $\h^{i,j}_{\rm dR}(X):=\h^i(X_{\rm an},\Omega^{<j})$ be the truncated analytic De Rham cohomology of $X$.  Then there is a commutative diagram
	\begin{equation*}
	\xymatrix{
	\h^i(X)\ar[r]\ar[dr]& \h^i(X)_\C/F^i\ar[r]&\h^i(X)_\C/F^{i-1}\ar[r]&\cdots\ar[r]&\h^i(X)_\C/F^1\\
	&\h^{i,i}_{\rm dR}(X)\ar[u]^{\pi_i}\ar[r]&\h^{i,i-1}_{\rm dR}(X)\ar[u]^{\pi_{i-1}}\ar[r] &\cdots\ar[r]&\h^{i,1}_{\rm dR}(X)\ar[u]^{\pi_1}
	 }
	\end{equation*}
	where the $\C$-linear maps $\pi_j$ are surjective. This diagram is the  formal Hodge structure $\h^{i,i}_\sharp(X)$ (or simply  $\h^{i}_\sharp(X)$).\\
	Note that this definition is compatible with the theory of formal Hodge structures of level $\le 1$ developed by L. Barbieri-Viale (See \cite{bv:fht}). He defined $\h^1_\sharp(X)$ as the generalized Hodge realization of $\Pic^0(X)$, i.e. $\h^1_\sharp(X):=T_{\oint}(\Pic^0(X))$ which is explicitly represented  by the diagram
	\begin{equation*}
	\xymatrix{
\h^1(X)\ar[r]\ar[dr]	& \h^1(X)_\C/F^1\\
	&\h^{1,1}_{\rm dR}(X)\ar[u]^{\pi_1}
	 }	
	\end{equation*}

 As an application of this theory  we can express the Albanese variety of Esnault, Srinivas and Viehweg (\cite{esnault.srinivas.viehweg}) using ext-groups. Precisely let $X$ be a proper, irreducible, algebraic scheme over $\C$. Let $d=\dim X$ and denote by $\h^{2d-1,d}_\sharp(X)$ the formal Hodge structure represented by the following diagram
	\begin{equation*}
	\xymatrix{
	\h^{2d-1}(X)\ar[dr]^h\ar[r]& \h^{2d-1}(X)_\C/F^d\ar[r]&\cdots  \h^{2d-1}(X)_\C/F^1\\
	 & \h^{2d-1,d}_{\rm dR}(X)\ar[u]\ar[r]&\cdots \h^{2d-1,1}_{\rm dR}(X)\ar[u] \ .}
	\end{equation*}
	Then there is an isomorphism of complex Lie groups
	\[
		{\rm ESV}(X)_{\rm an}\iso \Ext^1_{{\sf FHS}_d}(\Z(-d),\h^{2d-1,d}_\sharp(X))
	\]
	where ${\rm ESV}(X)$ is the generalized Albanese of \cite{esnault.srinivas.viehweg}. Note that this formula generalizes the classical one 
	\begin{equation*}
		{\rm Alb}(X)_{\rm an}\iso \Ext^1_{\sf MHS}(\Z(-d),\h^{2d-1}(X)) 
	\end{equation*}
	which follows from the work of Carlson (See \cite{carlson}).
\subsection*{Acknowledgments}
The author would like to thank L. Barbieri-Viale for 
pointing his attention to this subject and  for helpful discussions. The author also thanks A. Bertapelle for many useful comments and suggestions.
\tableofcontents
\section{Formal Hodge Structures}
We simply call a \emph{formal group} a commutative group  of the form $H=H^o\x H_{\rm et}$ where $H_{\rm et}$ is a finitely generated abelian group and $H^o$ is a finite dimensional $\C$-vector space. We denote by $\sf FrmGrp$ the category with objects formal groups and \emph{morphisms} $f=(f^o,f_\et):H\to H'$, where  $f^o:H^o\to H'^o$ is $\C$-linear and $f_\et:H_\et \to H_\et'$ is $\Z$-linear.

We denote the category of mixed Hodge structures of level $\le l$  (i.e. of type $\{(n,m) |\ 0\le n,m \le l\}$) by ${\sf MHS}_l={\sf MHS}_l(0)$, for $l\ge 0$. Also we define the category ${\sf MHS}_l(n)$ to be the full sub-category of $\sf MHS$ whose objects are $H_\et\in \sf MHS$ such that $H_\et\ox \Z(-n)$ is in ${\sf MHS}_l(0)$.\\
Let ${\sf Vec}={\sf Vec}_1$ be the category of finite dimensional complex vector spaces and $n>0$ be an integer. We define the category  ${\sf Vec}_n$, as follows. The objects are   diagrams of $n-1$  composable arrows of $\sf{Vec}$ denoted by
\[
	V:\  V_n \stackrel{v_n}{\longrightarrow} V_{n-1}\stackrel{v_{n-1}}{\longrightarrow}V_{n-2}\to\cdots\to V_1 \ .
\]

Let $V$, $V'\in {\sf Vec}_n$,  a \emph{morphism}  $f:V\to V'$ is a family $f_{i}:V_{i}\to V_{i}'$ of $\C$-linear maps such that\\
	\makebox[\textwidth][c]{
	\xymatrix{
	V_{i+1} \ar[d]^{f_{i+1}}\ar[r]&    V_{i}\ar[d]^{f_{i}}\\
	V_{i+1}'\ar[r]&V_{i}'  
	}}
	is commutative for all $1\le i\le n$.
\begin{dfn}[level $=0$]
	We define the category of \dfni{formal Hodge structures of level $0$} (twisted by $k$), ${\sf FHS}_0(k)$ as follows: the objects are formal groups $H$ such that $H_\et$ is a pure Hodge structure of type $(-k,-k)$; morphism are maps of formal groups.
	
	Equivalently ${\sf FHS}_0(k)$ is the product category ${\sf MHS}_0(k)\x \sf Vec$.
\end{dfn}
\begin{dfn}[level $\le n$]\label{def:fhsn}
	Fix $n>0$ an integer. We define a \dfni{formal Hodge structure of level $\le n$} (or a \emph{$n$-formal Hodge structure}) to be the data of
	
	i) A   formal group $H$ (over $\C$) carrying a mixed Hodge structure on the \'etale component, $(H_\et,F,W)$, of level $\le n$. Hence we get  $F^{n+1}H_\C=0$ and $F^{0}H_\C=H_\C$, where $H_\C:=H_\et\ox \C$. 
	
	ii) A family of fin. gen. $\C$-vector spaces $V_i$, for $1\le i\le n$.
	
	iii) A commutative diagram of abelian groups
	\begin{equation*}
	\xymatrix{
	H_\et\ar[dr]^{h_\et}\ar[r]^c&H_\C/F^{n}\ar[r]&H_\C/F^{n-1}\ar[r]&\cdots\ar[r]&H_\C/F^1\\
H^o\ar[r]_{h^o}&V_{n}\ar[u]^{\pi_{n}}\ar[r]_{v_{n}}&V_{n-1}\ar[u]^{\pi_{n-1}}\ar[r]_{v_{n-1}}&\cdots\ar[r]&V_1\ar[u]^{\pi_1}  }
	\end{equation*}
	such that $\pi_i$, $h^o$ are $\C$-linear maps.\\[1ex]
	We denote this object by $(H,V)$ or $(H,V,h,\pi)$. Note that $V=\{V_n\to\cdots\to V_1\}$ can be viewed as an object of ${\sf Vec}_n$.\\ 
	The map $h=(h_\et,h^o):H\to V_n$ is called \emph{augmentation} of the given formal Hodge structure.	\\[1ex]
	A \emph{morphism} of $n$-formal Hodge structures is a pair $(f,\phi)$ such that: $f:H\to H'$ is a morphism of formal groups; $f$ induces a morphism of mixed Hodge structures $f_\et$; $\phi_i:V_i\to V_i'$ is a family of $\C$-linear maps; $\phi:V\to V'$ is a morphism in ${\sf Vec}_n$; $(f,\phi)$ are  compatible with the above structure, i.e. such that the following diagram  commutes
	\begin{equation*}
	\xymatrix{
	& &  H_\et'\ar[dr]^{h_\et'}\ar[r]&H_\C'/F\\
 H_\et\ar[dr]^{h_\et}\ar[r]\ar[rru]^{f_\et}&H_\C'/F\ar[rru]^{\bar{f}_\C}&(H')^o\ar[r]_{(h')^o}&V'\ar[u]_{\pi'}\\
	 H^o\ar[r]_{h^o}\ar[rru]^{ \qquad f^o}&V\ar[u]^{\pi}\ar[rru]_{\phi} 
	}
	\end{equation*}
	
	We denote this category by ${\sf FHS}_{n}={\sf FHS}_n(0)$.
\end{dfn}
\begin{rmk}
	Note that the commutativity of the diagram (iii) of the above definitions implies that the maps $\pi_i$ are surjective. In fact after tensor by $\C$ we get that the composition $\pi_n\circ h_\C$ is the canonical projection $H_\C\to H_\C/F^n$: hence $\pi_n$ is surjective. Similarly we obtain the surjectivity of $\pi_i$ for all $i$.   
\end{rmk}
\begin{exm}[Sharp cohomology of a curve]
	Let $U=X\setminus D$ be a complex projective curve minus a finite number of points. Then we get the following commutative diagram
	 \begin{equation*}
	\xymatrix{
	\h^1(U)\ar[dr]^{}\ar[r]&\h^1(U)_\C/F^1\\
\Ker(\h^{1,1}_{\rm dR}(X)\to \h^{1,1}_{\rm dR}(U))\ar[r]_{}&\h^{1,1}_{\rm dR}(X)\ar[u]^{\pi_1}  }
	\end{equation*}
	representing a formal Hodge structure of level $\le 1$.
\end{exm}
\begin{rmk}[Twisted fhs]
	In a similar way one can define the category ${\sf FHS}_n(k)$ whose object are represented by diagrams
	\begin{equation*}
	\xymatrix{
	H_\et\ar[dr]^{h_\Z}\ar[r]&H_\C/F^{n-k}\ar[r]&H_\C/F^{n-1-k}\ar[r]&\cdots\ar[r]&H_\C/F^{1-k}\\
H^o\ar[r]_{h^o}&V_{n-k}\ar[u]^{\pi_{n-k}}\ar[r]_{v_{n-k}}&V_{n-k-1}\ar[u]^{\pi_{n-k-1}}\ar[r]_{v_{n-k-1}}&\cdots\ar[r]&V_{1-k}\ar[u]^{\pi_{1-k}}  }
	\end{equation*}
	where $H_\et$ is an object of ${\sf MHS}_n(k)$.
	
	Hence the Tate twist $H_\et\mapsto H_\et\ox\Z(k)$ induces an equivalence of categories
	\[
	{	{\sf FHS}_n(0)\rightarrow {\sf FHS}_n(k)}\qquad (H,V)\mapsto (H(k),V(k))
	\]
	where $H(k)=H_\et\ox \Z(k)\x H^o$ and $V(k)$ is obtained by $V$ shifting the degrees, i.e. $V(k)_i=V_{i+k}$, for $1-k\le i \le n-k$.
\end{rmk}
\begin{exm}[Level $\le 1$]
	 According to the above definition a $1$-formal Hodge structure twisted by $1$ is represented by a diagram 
		\begin{equation*}
		\xymatrix{
	H_\et\ar[dr]^{h_\et}\ar[r]&H_\C/F^{0}\\
H^o\ar[r]_{h^o}&V_{0}\ar[u]^{\pi_{0}} }
		\end{equation*}
	where is $(H_\et,F,W)$ be a mixed Hodge structure of level $\le 1$ (twisted by $\Z(1)$), i.e. of type $[-1,0]\x [-1,0] \subset \Z^2$ (recall that this implies $F^1 H_\C=0$ and $F^{-1} H_\C=H_\C $). If we further assume that $H_\et$ carries a mixed Hodge structure such that $\gr^W_{-1}H_\et$ is polarized we get the category studied in \cite{bv:fht}.
\end{exm}
\begin{prp}[Properties of FHS]\label{prp:fhsproperties}
	i) The category ${\sf FHS}_n$ is an abelian category.
	
	ii) The forgetful functor $(H,V)\mapsto H$ (resp. $(H,V)\mapsto V$) is an exact functor with values in the category of formal groups (resp. the category ${\sf Vec}_n$).
	
	iii) There exists a full and thick embedding ${\sf MHS}_l(0)\to {\sf FHS}_l(0)$ induced by $(H_\et,F,W)\mapsto (H_\et,V_i=H_\C/F^i)$.
	
	iv) There exists a full and thick embedding ${\sf Vec}_l(0)\to {\sf FHS}_l(0)$ induced by $V\mapsto (0,V)$.
\end{prp}
\begin{proof}
	i) It follows from the fact that we can compute kernels, co-kernels and direct sum component-wise. 

ii) It follows by (i).

iii) Let $(f,\phi):(H_\et,H_\C/F)\to (H_\et',H_\C'/F)$ be a morphism in $\sf FHS_n$. Then by definition for any $1\le i\le n$ there is a commutative diagram
\begin{equation*}
\xymatrix{
 H_\C/F^i\ar[d]_{\id}\ar[r]^{\phi_i}&  H_\C'/F^i  \ar[d]^{\id}\\
H_\C/F^i\ar[r]_{\bar{f}_i} &H_\C'/F^i   }
\end{equation*}
where $\bar{f}_i(x + F^iH_\C)=f(x)+F^iH_\C'$ is the map induce by $f$: it is well defined because the morphisms of mixed Hodge structures are  strictly compatible w.r.t. the Hodge filtration. Hence $\phi$ is completely determined by $f$.

iv) It is a direct consequence of the definition of ${\sf FHS}_n$. 
\end{proof}
\begin{lmm}
	Fix $n\in \Z$. The following functor
	\[
		{{\sf MHS} \rightarrow  \sf Vec}\ ,\quad (H_\et,W,F)\mapsto H_\C/F^n
	\]
	is an exact functor.
\end{lmm}
\begin{proof}
	This follows from \cite[\S 1.2.10]{deligne:hodge2}.
\end{proof}

\subsection{Sub-categories of ${\sf FHS}_n$}
Let $(H,V)$ be a formal Hodge structure of level $\le n$. It can be visualized as a diagram
	\begin{equation*}
	\xymatrix{
	H_\et\ar[dr]^{h_\et}\ar[r]&H_\C/F^{n}\ar[r]&H_\C/F^{n-1}\ar[r]&\cdots\ar[r]&H_\C/F^1\\
H^o\ar[r]_{h^o}&V_{n}\ar[u]^{\pi_{n}}\ar[r]_{v_{n}}&V_{n-1}\ar[u]^{\pi_{n-1}}\ar[r]_{v_{n-1}}&\cdots\ar[r]&V_1\ar[u]^{\pi_1} \\
& V_n^o\ar[u]\ar[r]& V_{n-1}^o\ar[u]\ar[r]&\cdots\ar[r]&V_{1}^o\ar[u] }
	\end{equation*}
	where $V^o_i:=\Ker (\pi_i:V_i\to H_\C/F^i)$. We can consider the following $n$-formal Hodge structures
\begin{enumerate}
	\item $(H,V)_\et:=(H_\et,V/V^o)$, called the \emph{\'etale part} of $(H,V)$.
	\item $(H,V)_\x:=(H,V/V^o)$, where the augmentation $H\to H_\C/F^{n}=V_n/V_n^o$ is the composite $\pi_{n}\circ h$.
\end{enumerate}
	We say that $(H,V)$ is \emph{\'etale} (resp. \emph{connected}) if $(H,V)=(H,V)_\et$ (resp. $(H,V)_\et=0$).	Also we say that $(H,V)$ is \emph{special} if $h^o:H^o\to V_n$ factors through $V^o_n$. We will denote by ${\sf FHS}_{n,\et}$ (resp. ${\sf FHS}_{n}^o$, ${\sf FHS}_{n}^s$) the full sub-category of ${\sf FHS}_{n}$ whose objects are \'etale (resp. connected, special). Note that by construction the category of \'etale formal Hodge structure ${\sf FHS}_{n,\et}$ is equivalent to ${\sf MHS}_n$, by abuse of notation we will identify these two categories.
\begin{prp}[Canonical Decomposition]\label{prp:fhssplittings}
	i) Let $(H,V)\in {\sf FHS}_n$ ($n>0$), then there are two canonical exact sequences
	\[
		0\to (0,V^o)\to (H,V)\to (H,V)_\x\to 0\quad; 0\to (H,V)_\et\to (H,V)_\x\to (H^o,0)\to 0
	\]
	ii) The augmentation $h^o:H^o\to V_{n}$ factors trough $V_{n}^o$ $\iff$ there is a canonical exact sequence
	\[
		0\to (H,V)^o\to (H,V)\to (H,V)_\et\to 0 
	\]
	where $(H,V)^o:=(H^o,V^o)$.
\end{prp}
\begin{proof}
		i) Let $(0,\theta):(0,V^o)\to (H,V)$ be the canonical inclusion. By \ref{prp:fhsproperties} $\Coker (0,\theta)$ can be calculated in the product category ${\sf FrmGrp}\x {\sf Vec}_n$, i.e.  $\Coker (0,\theta)=\Coker 0\x \Coker \theta= H\x V/V^o$ and the augmentation $H\to H_\C/F^n$ is the composition $H \stackrel{h}{\rightarrow}V_n \stackrel{\pi_n}{\rightarrow}H_\C/F^n$.
		
		For the second exact sequence  consider the natural projection $p^o:H\to H^o$. It induces a morphism $(p^o,0):(H,V)_\x\to (H^o,0)$. Using the same argument as above we get $\Ker (p^o,0)=\Ker p^o\x \Ker 0= H_\et\x V/V^0$ as an object of ${\sf FrmGrp}\x {\sf Vec}_n$. From this follows the second exact sequence.
		
		ii) By the  definition of a morphism of formal Hodge structures (of level $\le n$) we get that the canonical map, in the category ${\sf FrmGrp}\x {\sf Vec}_n$, $(p_{\Z}, \pi): H\x V\to H_\et \x V/V^o$ induces a morphism of formal Hodge structures $\iff$ the following diagram commutes
		\begin{equation*}
		\xymatrix{
		H \ar[d]_{h}\ar[r]^{p_\Z}&    H_\Z\ar[d]^{}\\
		V_n\ar[r]_{\pi_n} & H_\C/F^n  }
		\end{equation*}
		i.e. $\pi_n h(x,y)=y \mod F^nH_\C$ for all $x\in H^o,\ y\in H_\et$ $\iff$ $h^o(x)=0.$
\end{proof}
\begin{rmk}
	With the above notations  consider the map $(p^o,0):H\x V\to H^o\x 0$. Note that this is a morphism of formal Hodge structure $\iff$ $V^0=0$ $\iff$ $(H,V)=(H,V)_\x$.
\end{rmk}
\begin{rmk}
	For $n=0$ we can also use the same definitions, but the situation is much more easier. In fact a formal structure of level $0$ is just a formal group $H$, hence  there is a split exact sequence
	\[
		0\to H^o\to H\to H_{\rm et}\to 0
	\]
	in ${\sf FHS}_0(0)$.
\end{rmk}
\begin{crl}
	Let ${\frak K}_0({\sf FHS_n})$ be the Grothendieck group (see \cite[Def. A.4]{peters-steenbrink}) associated to the abelian category ${\sf FHS}_n$. Then
	\begin{align*}
		{\frak K}_0({\sf FHS_n}) & = {\frak K}_0({\sf Vec})\x {\frak K}_0({\sf Vec}_n)\x {\frak K}_0({\sf MHS}_n)\\
		&\iso \{(f,g)\in \Z[t]\x \Z[u,v]|\ \deg_t f,\deg_u g, \deg_v g\le n\ , \ g(u,v)=g(v,u) \}
	\end{align*}
\end{crl}
\begin{proof}
	It follows easily by (i) of \ref{prp:fhssplittings}.
\end{proof}
By \ref{prp:fhsproperties} there exists a canonical embedding ${\sf MHS}_n\subset{\sf FHS}_n$ (resp. ${\sf Vec}_n\subset{\sf FHS}_n$). It is easy to check that this embedding gives, in the usual way, a full embedding when passing to the associated homotopy categories, i.e. 
\begin{equation}\label{eq:kmhsinkfhs}
	K({\sf MHS}_n)\subset K({\sf FHS}_n)\ ,\qquad\text{resp.}\ K({\sf Vec}_n)\subset K({\sf FHS}_n)\ .
\end{equation}
With the following lemma we can prove that we have an embedding when passing to the associated derived categories.
\begin{lmm}\label{lmm:fullloc}
	Let $\sf A'\subset A$ be a full embedding of categories. Let $S$ be a multiplicative system in $\sf A$ and $S'$ be its restriction to $\sf A'$. Assume that one of the following conditions
	
	i) For any $s:A'\to A$ (where $A'\in \sf A'$, $A\in \sf A$, $s\in S$) there exists a morphism $f:A\to B'$ such that $B'\in \sf A'$ and $f\circ s\in S$.
	
	ii) The same as (i) with the arrow reversed.\\
	Then the localization ${\sf A}_{S'}'$ is a full sub-category of ${\sf A}_S$.
\end{lmm}
\begin{proof}
	\cite[1.6.5]{kashiwara.schapira:som}.
\end{proof}
\begin{prp}
	There is a full embedding of  categories $D({\sf MHS}_n)\subset D({\sf FHS}_n)$ (resp. $D({\sf Vec}_n)\subset D({\sf FHS}_n)$).
\end{prp}
\begin{proof}
	We will prove only the case involving  ${\sf MHS}_n$, the other one is similar. First note that similarly to \eqref{eq:kmhsinkfhs} there is a full embedding $K({\sf FHS}_{n,\x})\subset K({\sf FHS}_{n})$, where ${\sf FHS}_{n,\x}$ is the full sub-category of ${\sf FHS}_n$ with objects $(H,V)$ such that $(H,V)=(H,V)_\x$ (See \ref{prp:fhssplittings}). Now using $(i)$ of lemma \ref{lmm:fullloc} and the first exact sequence of \ref{prp:fhssplittings} we get a full embedding $D({\sf FHS}_{n,\x})\subset D({\sf FHS}_{n})$.\\
	Then consider the canonical embedding ${\sf MHS}_n\subset{\sf FHS}_{n,\x}$. Again we get a full embedding of triangulated categories $K({\sf MHS}_{n})\subset K({\sf FHS}_{n,\x})$. Now using $(ii)$ of lemma \ref{lmm:fullloc} and the second exact sequence of \ref{prp:fhssplittings} we get a full embedding $D({\sf FHS}_{n,\x})\subset D({\sf FHS}_{n})$. 
\end{proof}
\subsection{Adjunctions}
\begin{prp}
	The following adjunction formulas hold
	
	i) $\Hom_{\sf MHS}(H_\et,H_\et')\iso \Hom_{{\sf FHS}_n}((H,V),(H_\et',H_\C'/F))$ for all $(H,V)\in {\sf FHS}_n^s$ (i.e. special), $H_\et'\in {\sf MHS}_n$.
	
	ii) $\Hom_{{\sf FHS}_n}((H^o,V),(H',V'))\iso \Hom_{{\sf FHS}_n}((H^o,V),((H')^o,(V')^o)) $ for all $(H^o,V)\in {\sf FHS}_n^{o}$ (i.e. connected), $(H',V')\in {\sf FHS}_n^s$.
\end{prp}
\begin{proof}
	The proof is straightforward. Explicitly:
	i) Let $(H,V)\in {\sf FHS}_n^{s}$, $H_\et'\in {\sf MHS}_n$. By definition a morphism $(f,\phi)\in \Hom_{{\sf FHS}_n}((H,V),(H_\et',H_\C'/F))$ is a morphism of formal groups $f:H\to H'$ such that $f_\et$ is a morphism of mixed Hodge structures, hence $f=f_\et$, and $\phi:V\to H_\C'/F$ is subject to the condition $f/F\circ \pi=\phi$. Then the association $(f,\phi)\mapsto f_\et\in \Hom_{\sf MHS}(H_\et,H_\et')$ is an isomorphism.
	
	ii) Let $(H^o,V)\in {\sf FHS}_n^{o}$, $(H',V')\in {\sf FHS}_n^{s}$.\\
	 A morphism $(f,\phi)$ in $\Hom_{{\sf FHS}_n}((H^o,V),(H',V'))$ is of the form $f=f^o:H^o\to (H')^o$, $\phi:V\to V'$ must factor through $(V')^o$ because $\pi'\circ\phi=\pi\circ f/F=0$. 
\end{proof}
\subsection{Different levels}
Any mixed Hodge structure of level $\le n$ (say in ${\sf MHS}_n(0)$) can also be viewed as an object of ${\sf MHS}_m(0)$ for any $m>n$. This give a sequence of full embeddings
\[
	{\sf MHS}_0\subset {\sf MHS}_1\subset\cdots\subset \sf MHS
\] 
In this section we want to investigate the analogous situation in the case of formal Hodge structures. 

Consider the two functors $\iota, \eta:{\sf Vec}_n\rightarrow {\sf Vec}_{n+1}$	defined as follows
\begin{align*}
	\iota(V):\ & \quad \iota(V)_{n+1}=V_n \stackrel{\id}{\rightarrow}\iota(V)_{n}=V_n \stackrel{v_n}{\rightarrow}\cdots\to V_1 \\
	\eta(V):\ & \quad \eta(V)_{n+1}=0 \stackrel{0}{\rightarrow}\iota(V)_{n}=V_n \stackrel{v_n}{\rightarrow}\cdots\to V_1 
\end{align*}
\begin{prp}\label{prp:veciota}
	The functors $\iota, \eta$ are full and faithful. Moreover the essential image of $\iota$ (resp. $\eta$)  is a thick sub-category\footnote{By thick we mean a sub-category closed under kernels, co-kernels and extensions}.
\end{prp}
\begin{proof}
	To check that $\iota,\eta$ are embeddings it is straightforward. We prove that the essential image of $\iota$ (resp. $\eta$) is closed under extensions only in case $n=2$ just to simplify the notations.
	
	First consider an extension of $\eta V$ by $\eta V'$ in ${\sf Vec}_3$
	\begin{equation*}
	\xymatrix{
		0\ar[r]& 0\ar[d]\ar[r]&V_3''\ar[d]\ar[r]& 0\ar[d]\ar[r]&0\\
	0\ar[r]& V_2'\ar[d]\ar[r]&V_2''\ar[d]\ar[r]& V_2\ar[d]\ar[r]&0\\ 
	0\ar[r]& V_1'\ar[r]&V_1''\ar[r]& V_1\ar[r]&0  }
	\end{equation*}
	then it follows that $V_3''=0$. 
	
	Now consider an extension of $\iota V$ by $\iota V'$ in ${\sf Vec}_3$
	\begin{equation*}
	\xymatrix{
		0\ar[r]& V_2'\ar[d]^\id\ar[r]&V_3''\ar[d]^v\ar[r]& V_2\ar[d]^\id\ar[r]&0\\
	0\ar[r]& V_2'\ar[d]\ar[r]&V_2''\ar[d]\ar[r]& V_2\ar[d]\ar[r]&0 \\
	0\ar[r]& V_1'\ar[r]&V_1''\ar[r]& V_1\ar[r]&0  }
	\end{equation*}
	Then $v$ is an isomorphism (by the snake lemma). It follows that $V''$ is isomorphic, in ${\sf Vec}_3$, to an object of $\iota {\sf Vec}_2$. To check that the essential image of $\iota$ (resp. $\eta$) is closed under kernels and cokernels is straightforward.
\end{proof}
\begin{rmk}
	The category of complexes of objects of $\sf Vec$ concentrated in degrees $1,...,n$ is a full sub-category of ${\sf Vec}_n$. Moreover the embedding induces an equivalence of categories for $n=1$ and $2$, but for $n>2$ the embedding is not thick.
\end{rmk}

\begin{exm}[${\sf FHS}_1\subset {\sf FHS}_2$]\label{exm:iotak}
	The basic construction is the following: 
		let $(H,V)$ be a $1$-fhs, we can associate a $2$-fhs $(H',V')$ represented by a diagram of the following type
			\begin{equation*}
			\xymatrix{
		H_\et'\ar[dr]_{h_\Z'}\ar[r]&H_\C'/F^{2}\ar[r]&H_\C'/F^{1}\\
		(H')^o\ar[r]_{(h')^o}&V_{2}'\ar[u]_{\pi_{2}'}\ar[r]_{v_{2}'}&V_{1}'\ar[u]_{\pi_{1}'} }
			\end{equation*}
		Take $H_\et':=H_\et$, then $H_\C'/F^{2}=H_\C$, $H_\C'/F^{1}=H_\C/F^{1}$ and the augmentation $h_\et'$ is the canonical inclusion; let $V_1':=V_1$, $\pi_1':=\pi_1$ and define $V_2',\ \pi_2',\ v_2'$ via fiber product
		\begin{equation*}
		\xymatrix{
		V_2' \ar[d]_{v_2'}\ar[r]^{\pi_2'}&    H_\C\ar[d]^{}\\
		V_1\ar[r]_{\pi_1} & H_\C/F^1  }
		\end{equation*}	
		Hence $V_2'$ fits in the following exact sequences
		\[
			0\to F^1H_\C\to V_2'\to V_1\to 0\quad ;\quad 	0\to V_1^0\to V_2'\to H_\C\to 0\ .
		\]
		Finally we define $(h')^o:(H')^o\to V_2'$ again via fiber product
		\begin{equation*}
		\xymatrix{
		(H')^o \ar[d]_{}\ar[r]^{(h')^o}&    V_2'\ar[d]^{v_2'}\\
		H^o\ar[r]_{h^o} & V_1  }
		\end{equation*}
		hence we get the following exact sequence
		\[
			0\to F^1H_\C\to 	(H')^o\to 	H^o\to 0\ .
		\]

	By induction is easy to extend this construction. We have the following result.
\end{exm}

	\begin{prp}\label{prp:difflevels}
		Let $n,k>0$. Then there exists a faithful functor
		\[
			\iota=\iota_k:{\sf FHS}_n\rightarrow {\sf FHS}_{n+k}
		\]
		Moreover $\iota$ induces an equivalence between ${\sf FHS}_n$ and the sub-category of ${\sf FHS}_{n+k}$ whose objects are $(H,V)$ such that
		
		a)  $H_\et$ is of level $\le n$. Hence $F^{n+1}H_\C=0$ and $F^{0}H_\C=H_\C$.
		
		b) $V_{n+i}=V_{n+1}$ for $1\le i\le k$.
		
		c) There exists a commutative diagram  with exact rows
		\begin{equation*}
		\xymatrix{
		F^n\ar[r]&H_\C\ar[r]&H_\C/F^n\\
		F^n\ar[rd]_\alpha\ar[u]^\id\ar[r] &V_{n+1}\ar[u]^{\pi_{n+1}}\ar[r]^{v_{n+1}}& V_n\ar[u]^{\pi_n}\\
		& H^o\ar[u]^{h^o}}
		\end{equation*}
		where  $\alpha$ is a $\C$-linear map.
		
		And morphisms are those in ${\sf FHS}_{n+k}$ compatible w.r.t. the diagram in $(c)$.
	\end{prp}
	\begin{proof}
		The construction of $\iota_k$ is a generalization of that in \ref{exm:iotak}. We have $\iota_k=\iota_1\circ \iota_{k-1}$, hence it is enough to define $\iota_1$ which is the same as in \ref{exm:iotak} up to a change of subscripts: $n=1$, $n+1=2$.
		
		To prove the equivalence we define a quasi-inverse: Let $(H',V')\in {\sf FHS}_{n+1}$ and satisfying $a,b,c$ and $\alpha :F^nH_\C'\to (H')^o$ as in the proposition.\\ Define $(H,V)\in {\sf FHS}_n $ in the following way: $H=H'/\alpha(F^nH_\C')$; $V_i=V_i'$ for all $1\le i\le n$; $h: H'/\alpha(F^nH_\C') \stackrel{\bar{h'}}{\longrightarrow} V_{n+1}' \stackrel{v_{n+1}'}{\longrightarrow} V_n'=V_n$, where $\bar{h'}=(h_\et',(h')^o \mod F^n)$.
	\end{proof}
	\begin{prp}
		Let $n,k>0$ and denote by $\iota_k{\sf FHS}_n\subset {\sf FHS}_{n+k}$ the essential image of ${\sf FHS}_n$ (See the previous proposition). Then $\iota_k{\sf FHS}_n\subset {\sf FHS}_{n+k}$ is an  abelian  (not full) sub-category closed under kernels, cokernels and extensions.  
	\end{prp}
	\begin{proof}
		Straightforward.
	\end{proof}
	\begin{rmk}
	Note that $\iota_k{\sf FHS}_n\subset {\sf FHS}_{n+k}$ it is not closed  under sub-objects.
	\end{rmk}
	\begin{rmk}
		Let ${\sf FHS}_n^{prp}$ be the full sub-category of ${\sf FHS}_n$ whose objects are formal Hodge structures $(H,V)$ with $H^o=0$\footnote{The superscript $prp$ stands for proper. In fact the sharp cohomology objects (\ref{def:sharpcoho}) of a proper variety have this property.}. Then $\iota_k$ induces a full and faithful functor 
		\[
			\iota=\iota_k:{\sf FHS}_n^{prp}\rightarrow {\sf FHS}_{n+k}^{prp}
		\]
		
		Moreover $\iota_k{\sf FHS}_n^{prp}\subset {\sf FHS}_{n+k}^{prp}$ is an abelian thick sub-category.
	\end{rmk}
\begin{exm}[Special structures]\label{exm:tauk}
	For special structures it is natural to consider the following construction, similar to $\iota_k$ (Compare with \ref{exm:iotak}). Let $(H,V)$ be a formal Hodge structures of level $\le 1$. Define $\tau(H,V)=(H,V')$ to be the formal Hodge structure of level $\le 2$ represented by the following diagram
		\begin{equation*}
		\xymatrix{
	H_\et\ar[dr]_{}\ar[r]&H_\C\ar[r]^{h_\C}&H_\C/F^{1}\\
	H^o\ar[r]_{(h')^o}&V_{2}'\ar[u]_{\pi_{2}'}\ar[r]_{v_{2}'}&V_{1}\ar[u]_{\pi_{1}} }
		\end{equation*}
		where $V_2',\ v_2',\ (h')^o$ are defined via fiber product as follows
		\begin{equation*}
		\xymatrix{
		H^o \ar@/_/[ddr]_{h^o} \ar@/^/[drr]^{0}
		\ar@{.>}[dr]|-{(h')^o}             \\
		& V_2'\ar[d]_{v_2'} \ar[r]^{\pi_2'} &  H_\C\ar[d]^{} \\
		& V_1\ar[r]_{\pi_1} &H_\C/F^1 
		}
		\end{equation*}
		Note that the commutativity of the external square is equivalent to say that $(H,V)$ is special. Hence this construction cannot be used for general formal Hodge structures.
\end{exm}
\begin{prp}
	 Let $n,k>0$ integers. Then there exists a full and faithful functor
	\[
		\tau=\tau_k:{\sf FHS}^{s}_n\rightarrow  {\sf FHS}^{s}_{n+k} 
	\]
	Moreover the essential image of $\tau_k$, $\tau_k{\sf FHS}^{spc}_n$, is the full and thick abelian sub-category of ${\sf FHS}^{spc}_{n+k}$ with objects $(H,V)$ such that
	
		a) $H_\et$ is of level $\le n$. Hence $F^{n+1}H_\C=0$ and $F^{0}H_\C=H_\C$. 
		
		b) $V_{n+i}=V_{n+1}$ for $1\le i\le k$.
		
		c) $V_{n+1}=H_\C\x_{H_\C/F^n}V_n$.
\end{prp}
\begin{proof}
	Note that $\tau_k=\tau_1\circ \tau_{k-1}$, hence is enough to construct $\tau_1$. Let $(H,V)$ be a special formal Hodge structure of level $\le n$, then $\tau_1(H,V)$ is defined as in \ref{exm:tauk}  up to change the sub-scripts $n=1,\ n+1=2$.
	
	 To prove the equivalence it is enough to construct a quasi-inverse of $\tau_1$. Let $(H',V')$ be a special formal Hodge structure of level $\le n$ satisfying the conditions $a,b,c$ of the proposition, then define $(H,V)\in {\sf FHS}_n$ as follows: $H:=H'$; $V_i:=V_i'$ for all $1 \le i\le n$; $h=v_{n+1}'\circ h'$.
	
	Thickness follows directly from the exactness of the functors
	\[
		(H,V)\mapsto H_\et\ ,\quad (H,V)\mapsto V^o\ .
	\]
\end{proof}
\begin{rmk}
	The functors $\tau_k, \iota_k$ agree on the full sub-category of ${\sf FHS}_n$ formed by $(H,V)$ with $H^o=0$.
\end{rmk}
\section{Extensions in ${\sf FHS}_n$}
\subsection{Basic facts}
\begin{exm} We describe the ext-groups for ${\sf Vec}_2$. We have the following isomorphism
		\[
			\phi:\Ext^1_{{\sf Vec}_2}(V,V')\stackrel{\sim}{\rightarrow} \Hom_{\sf Vec}(\Ker v,\Coker v')
		\]
		Explicitly $\phi$ associates to any extension class the $\Ker$-$\Coker$ boundary map of the snake lemma. To prove it is an isomorphism we argue as follows. The abelian category ${\sf Vec}_2$   is equivalent to the full sub-category $C'$ of $C^b({\sf Vec})$ of complexes concentrated in degree $0,1$. Hence the group of classes of extensions is isomorphic. Now let $a:A^0\to A^1$, $b:B^0\to B^1$ be two complexes of objects of $\sf Vec$. Then we have
		\[
			\Ext^1_{C'}(A^\bullet,B^\bullet)=\Ext^1_{C^b({\sf Vec})}(A^\bullet,B^\bullet)=\Hom_{D^b({\sf Vec})}(A^\bullet,B^\bullet[1])
		\]
		because $C'$ is a thick sub-category of  $C^b({\sf Vec})$. 

		The category $\sf Vec$ is of cohomological dimension $0$, then $a:A^0\to A^1$  is quasi-isomorphic to $\Ker a \stackrel{0}{\rightarrow}\Coker a$, similarly for $B^\bullet$. It follows that 
		\begin{align*}
				\Hom_{D^b({\sf Vec})}(A^\bullet,B^\bullet[1])=&\Hom_{D^b({\sf Vec})}(\Ker a[0]\oplus \Coker a[-1], \Ker b[1]\oplus \Coker b[0])\\
				&=\Hom_{\sf Vec}(\Ker a,\Coker b)\ .
		\end{align*}
		As a corollary we obtain that $\Ext^1_{{\sf Vec}_2}(V,-)$ is a right exact functor and this is a sufficient condition for the vanishing of $\Ext^i_{{\sf Vec}_2}(,-)$ for $i\le 2$ (i.e. ${\sf Vec}_2$ is a category of cohomological dimension $1$.).
		
\end{exm}
\begin{exm}
	 The category ${\sf Vec}_3$  is of cohomological dimension $1$. We argue as in \cite{extlaumot}. Let $V$ be an object of  ${\sf Vec}_3$, we define the following increasing filtration
	\[
		W_{-2}=\{0\to0\to V_1\}\ ;\ W_{-1}=\{0\to V_2\to V_1\}\ ;\ W_0=V
	\]
	Note that morphisms in  ${\sf Vec}_3$ are compatible w.r.t. this filtration. To prove that $\Ext^2_{ {\sf Vec}_3}(V,V')=0$ it is sufficient to show that $\Ext^2_{{\sf Vec}_3 }(\gr_i^WV,\gr_j^WV')=0$ for $i,j=-2,-1,0$ (just use the short exact sequences induced by $W$, cf. \cite[Proof of 2.5]{extlaumot}). We prove the case $i=0$, $j=-2$ leaving to the reader the other cases (which are easier, cf. \cite[2.2-2.4]{extlaumot}).\\
	Let $\gamma\in \Ext^2_{{\sf Vec}_3 }(\gr_0^WV,\gr_{-2}^WV')=0$, we can represent $\gamma$ by an exact sequence in ${\sf Vec}_3$ of the following type
	\[
		0\to \gr_{-2}^WV'\to A\to B\to \gr^W_0V\to 0
	\] 
	Let $C=\Coker(\gr_{-2}^WV'\to A)=\Ker(B\to \gr^W_0V)$, then $\gamma =\gamma_1\cdot \gamma_2$ where $\gamma_1\in \Ext^1_{{\sf Vec}_3 }(C,\gr_{-2}^WV')$, $\gamma_2\in \Ext^1_{{\sf Vec}_3 }(\gr_0^WV,C)$. Arguing as in \cite[2.4]{extlaumot} we can suppose that $C=\gr_{-1}^WC$, hence
	\[
		\gamma_1=[0\to\gr_{-2}^WV'\to A\to \gr_{-1}^WC\to 0]\ ,\  \gamma_2=[0\to\gr_{-1}^WC\to B\to \gr_{0}^WV\to 0]
	\]
	It follows that $A=\{0\to C_2\xrightarrow{f_1} V'_1\}$, $B=\{V_3\xrightarrow{f_2} C_2\to 0\}$ for some $f_1,f_2$. Now consider $D=\{V_3\xrightarrow{f_2}C_2\xrightarrow{f_1} V'_1\}\in {\sf Vec}_3$, then it is easy to check that
	\[
		\gamma_1=[0\to W_{-2}D\to W_{-1}D \to \gr^W_{-1}D\to 0]\ ,\ \gamma_2=[0\to \gr_{-1}D\to W_0D/W_{-2}D\to \gr_0^WD\to 0]
	\] 
	By \cite[Lemma 2.1]{extlaumot} $\gamma=0$.
\end{exm}
\begin{prp}
	Let $H_\et$ be a mixed Hodge structure of level $\le n$: we consider it as an \'etale formal Hodge structure. Let $(H',V')$  be be a formal Hodge structure of level $\le n$ (for $n>0$). Then
	
	i) There is a canonical isomorphism of abelian groups
	\[
		\Ext_{\sf MHS}^1(H_\et,H_\et')\iso \Ext_{{\sf FHS}_n}^1(H_\et,(H',V'/{V'}^o))\ .
	\]
	
	ii) For any $i\ge 2$ there is a canonical isomorphism 
	\[
		\Ext_{{\sf FHS}_n}^i(H_\et,(H',V'/{V'}^o))\iso \Ext_{{\sf FHS}_n}^i(H_\et,(H'^o,0))\ .
	\]
\end{prp}
\begin{proof}
	This follows easily by the computation of the long exact sequence obtained applying $\Hom_{\sf FHS_n}(H_\Z,-)$ to the short exact sequence
	\[
		0\to (H',V')_{\rm et}\to (H',V')_\x\to (H'^o,0)\to 0 \ .
	\]
\end{proof}
\begin{prp}\label{prp:extfhs}
	The forgetful functor $(H,V)\mapsto H_\et$ induces a  surjective morphism of abelian groups
\[
	\gamma:\Ext_{{\sf FHS}_n}^1((H,V),(H',V'))\rightarrow \Ext_{\sf MHS}^1(H_\et,H_\et')
\]
for any $(H,V),\ (H',V')$ with $H_\et,H_\et'$ free.
\end{prp}
\begin{proof}
Recall the extension formula for mixed Hodge structures is  (see \cite[I \S 3.5]{peters-steenbrink})
\begin{equation}\label{eq:extmhs}
	\Ext_{\sf MHS}^1(H_\et,H_\et')\iso \frac{W_0\iHom(H_\et,H_\et')_\C}{F^0\cap W_0 (\iHom(H_\et,H_\et')_\C)+ W_0\iHom(H_\et,H_\et')_\Z}
\end{equation}
		more precisely we get that any extension class can be represented  by $\tilde{H}_\et=(H_\et'\oplus H_\et,W,F_\theta)$ where the weight filtration is the direct sum $W_i H_\et'\oplus W_i H_\et$ and $F^i_\theta:=F^i H_\et'+ \theta(F^i H_\et) \oplus F^i H_\et$, for some $\theta\in W_0\iHom(H_\et,H_\et')_\C$. It follows that $\tilde{H}_\C/F^i_\theta=H_\C'/F^i\oplus H_\C/F^i$. Then we can consider the formal Hodge structure of level $\le n$ $(\tilde{H},\tilde{V})$ defined as follows: $\tilde{H}_\et=(H_\et'\oplus H_\et,W,F_\theta)$ as above; $\tilde{H}^o:=(H')^o\oplus H^o$; $\tilde{V}_i:=V_i'\oplus V_i$, $\tilde{v}_i:=(v_i',v_i)$; $\tilde{h}=(h',h)$. Then it easy to check that $(\tilde{H},\tilde{V})\in \Ext^1_{{\sf FHS}_n}((H',V'),(H,V))$ and $\gamma(\tilde{H},\tilde{V})=(H_\et'\oplus H_\et,W,F_\theta)$.
\end{proof}
\begin{exm}[Infinitesimal deformation] Let $f:\widehat{X}\to \Spec \C[\epsilon]/(\epsilon^2)$ a smooth and projective morphism. Write $X/\C$ for the smooth and projective variety corresponding to the special fiber, i.e. the fiber product
	\begin{equation*}
	\xymatrix{
	X \ar[d]_{}\ar[r]^{}&    \widehat{X}\ar[d]^{f}\\
	\Spec \C\ar[r]_{} & \Spec \C[\epsilon]/(\epsilon^2)  }
	\end{equation*}
	then (see  \cite[2.4]{bloch.srinivas:ehs}) for any $i,n$ there is a commutative diagram with exact rows
	\begin{equation*}
	\xymatrix{
	0\ar[r]& \h^{n-i+1}(X_{\rm an},\Omega^{i-1})\ar[d]^0\ar[r]&\h^n(\widehat{X}_{\rm an},\Omega^{<i})\ar[d]\ar[r]& \h^n(X_{\rm an},\C)/F^{i}\ar[d] \ar[r]&0\\
	0\ar[r]& \h^{n-i+2}(X_{\rm an},\Omega^{i-2})\ar[r]&\h^n(\widehat{X}_{\rm an},\Omega^{<i-1})\ar[r]& \h^n(X_{\rm an},\C)/F^{i-1} \ar[r]&0 }
	\end{equation*}

	Hence there is an extension of formal Hodge structures of level $\le n$
	\[
		0\to (0, V)\to (\h^n(X),\h^{n,*}_{\rm dR}(\widehat{X})) \to \h^n(X)\to 0
	\]
	 with $V_i=\h^{n-i+1}(X_{\rm an},\Omega^{i-1})$ and $v_i=0$.
\end{exm}
\begin{rmk}
	It is well known that the groups $\Ext^i(A,B)$ vanish in category of mixed Hodge structures for any $i>1$. It is natural to ask if the groups $\Ext^i_{\sf FHS_n}((H,V),(H',V'))$ vanish for $i>n$ (up to torsion).
	In particular Bloch and Srinivas raised a similar question for special formal Hodge structures (cf. \cite{bloch.srinivas:ehs}).
	 
	The author  answered positively this question for $n=1$ in \cite{extlaumot}.
\end{rmk}
\subsection{Formal Carlson theory}
\begin{prp}
	 Let $A,B$ torsion-free mixed Hodge structures. Suppose $B$ pure of weight $2p$ and $A$ of weights $\le 2p-1$. There is a commutative diagram of complex Lie group
	\begin{equation*}
	\xymatrix{
	  \Ext^1_{\sf MHS}(B,A) \ar[dr]_{i^*}\ar[r]^\gamma&\Hom_{\Z}(B_\Z^{p,p},J^p(A))\\
	& \Ext^1_{\sf MHS}(B_\Z^{p,p},A)\ar[u]^{\bar{\gamma}}}
	\end{equation*}
	where $\bar{\gamma}$ is an isomorphism;  $i^*$ is the surjection induced by the inclusion $i:B_\Z^{p,p}\to B$. 
\end{prp}
\begin{proof}
	This follows easily from the explicit formula \ref{eq:extmhs}. The construction of $\gamma,\ \bar{\gamma}$ is given in the following remark. Then choosing a basis of $B^{p,p}_\Z$ it is easy to check that $\bar{\gamma}$ is an isomorphism. 
\end{proof}
\begin{rmk}
	 i) Let $\{b_1,...,b_n\}$ a $\Z$-basis of $B_\Z^{p,p}$, then $\Hom_{\Z}(B_\Z^{p,p},J^p(A))\iso \oplus_{i=1}^n J^p(A)$ which is a complex Lie group.

	ii) Explicitly $\gamma$ can be constructed as follows. Let $x\in \Ext^1_{\sf MHS}(B,A)$ represented by the extension \[
		0\to A\to H\to B\to 0
	\]
	then apply $\Hom_{\sf MHS}(\Z(-p),-)$ to the above exact sequence and consider the boundary of the associated long exact sequence
	\[
		\cdots \to \Hom_{\sf MHS}(\Z(-p),B) \stackrel{\partial_x}{\longrightarrow} \Ext^1_{\sf MHS}(\Z(-p),A)\to \cdots
	\]
	Note that $\partial_x$ does not depend on the choice of the representative of $x$; $\Hom_{\sf MHS}(\Z(-p),B)=B^{p,p}_\Z$; $J^p(A)=\Ext^1_{\sf MHS}(\Z(-p),A)$.

	Hence we can define $\gamma(x):=\partial_x\in \Hom_{\Z}(B_\Z^{p,p},J^p(A))$.

	iii) If the complex Lie group $J^p(A)$ is algebraic  then $\Hom_{\Z}(B_\Z^{p,p},J^p(A))$ can be identified with set of one motives of type
	\[
		u:B_\Z^{p,p}\rightarrow J^p(A)
	\]
	
\end{rmk}
\begin{dfn}[formal-p-Jacobian]
	Let $(H,V)$ be a formal Hodge structure of level $\le n$. Assume $H_\et$ is  a torsion free mixed Hodge structure. For $1\le p\le n$ the \dfni{$p$-th formal Jacobian of $(H,V)$} is defined
	as%
	$$
	J^p_\sharp(H,V):= V_p/H_\et. 
	$$
	where $H_\et$ acts on $V_p$ via the augmentation $h$. By construction there is an extension of abelian groups
	$$
	0\to V_p^0\to J^p_\sharp(H,V)\to J^p(H,V)\to 0
	$$
	where we define $J^p(H,V):=J^p(H_\et)=H_\C/(F^p+H_\et)$.
\end{dfn}
 
Note that that $J^p_\sharp(H,V)$ is a complex Lie group if the weights of $H_\et$ are $\le 2p-1$.
\begin{prp}\label{prp:genext}
	There is an extension of abelian groups
	\[
		0\to V_p^o\to \Ext^1_{{\sf FHS}_{p}}(\Z(-p),(H,V))\to \Ext^1_{\sf MHS} (\Z(-p),H_\et)\to 0
	\]
	for any $(H,V)$ formal Hodge structure of level $\le p+1$. In particular  if $H_\et$ has weights $\le 2p-1$ there is an extension
	\begin{equation}\label{eq:genext}
		0\to V_p^o\to \Ext^1_{{\sf FHS}_{p}}(\Z(-p),(H,V))\to J^p(H_\et)\to 0\ .
	\end{equation}
\end{prp}
\begin{proof}
	By \ref{prp:extfhs} there is a surjective map 
	\[
		\gamma :\Ext^1_{{\sf FHS}_{p}}(\Z(-p),(H,V))\to \Ext^1_{\sf MHS} (\Z(-p),H_\et)\ .
	\]
	Recall that $\Z(-p)$ is a mixed Hodge structure and here is considered as a formal Hodge structure of level $\le p$   represented by the following diagram
	\begin{equation*}
	\xymatrix{
	\Z\ar[dr]^{}\ar[r]^{}&0\ar[r]&\cdots\\
	  0\ar[r]_{h^o}&0\ar[u]\ar[r]&\cdots}
	\end{equation*}
	It follows directly from the definition of a morphism of formal Hodge structures that an element of $\Ker \gamma$ is a formal Hodge structure of the form $(H\x \Z(-p),H/F)$ represented by
		\begin{equation*}
		\xymatrix{		H_\et\x \Z\ar[dr]^{h_\et'}\ar[r]&H_\C/F^{n}\ar[r]&H_\C/F^{n-1}\ar[r]&\cdots\ar[r]&H_\C/F^1\\
	H^o\ar[r]_{h^o}&V_{n}\ar[u]^{\pi_{n}}\ar[r]_{v_{n}}&V_{n-1}\ar[u]^{\pi_{n-1}}\ar[r]_{v_{n-1}}&\cdots\ar[r]&V_1\ar[u]^{\pi_1}  }
		\end{equation*} 
	where the augmentation $h_\et'(x,z)=h_\et(x)+\theta (z)$ for some $\theta:\Z\to V_p^o$. The map $\theta$ does not depend on the representative of the class of the extension because $V_p$ and $\Z(-p)$ are fixed.
\end{proof}
\begin{exm}
	By the previous proposition for $p=1$ we get
	\[
		0\to V_1^o\to \Ext_{{\sf FHS}_1}^1(\Z(-1),(H,V))\to \Ext_{\sf MHS}^1(\Z(-1),H_\et) \to 0\ .
	\]
\end{exm}
\section{Sharp Cohomology}\label{cha:sc}
\begin{dfn}\label{def:sharpcoho}
	Let $X$ be a proper scheme over $\C$, $n>0$ and $1\le k\le n$. We define the \dfni{sharp cohomology object} $\h_\sharp^{n,k}(X)$ to be the $n$-formal Hodge structure represented by the following diagram
	\begin{equation*}
	\xymatrix{
	\h^n(X)\ar[dr]\ar[r]&\h^n(X)_\C/F^n\ar[r]&\cdots\ar[r]&\h^n(X)_\C/F^1\\
	&V^{n,k}_n(X)\ar[u]\ar[r]&\cdots\ar[r]&V^{n,k}_1(X)\ar[u]    }
	\end{equation*}
	where 
	\[
	V^{n,k}_i(X):=	\begin{cases}
	\h^{n,i}_{\rm dR}(X)& \text{if}\ 1\le i\le k\\
		\h^n(X)_\C/F^i\x_{\h^n(X)_\C/F^k}\h^{n,k}_{\rm dR}(X)& \text{if}\ k<i\le n
	\end{cases}
	\]
\end{dfn}

In the case $n=k$ we will simply write   $\h_\sharp^{n}(X)=\h_\sharp^{n,n}(X)$. This object  is represented explicitly by 
\begin{equation*}
\xymatrix{
\h^n(X_{\rm an},\Z)\ar[dr]\ar[r]&\h^n(X_{\rm an},\C)/F^n\ar[r]&\h^n(X_{\rm an},\C)/F^{n-1}\ar[r] &\cdots\ar[r]&\h^n(X_{\rm an},\C)/F^1\\
&\h^n(X_{\rm an},\Omega^{<n})\ar[u]\ar[r]& \h^n(X_{\rm an},\Omega^{<n-1})\ar[u]\ar[r]&\cdots\ar[r]&\h^n(X_{\rm an},\O)\ar[u]    }
\end{equation*}
\begin{exm}\label{exm:albcoho}
	Let $X$ be a proper scheme of dimension $d$ (over $\C$). Then $\h^{2d-1}(X)$ is a mixed Hodge structure satisfying $F^{d+1}=0$ and the sharp cohomology object $\h_\sharp^{2d-1,d}(X)$ is represented by
	\begin{equation*}
	\xymatrix{
	\h^{2d-1}(X)\ar[dr]\ar[r]&\h^{2d-1}(X)_\C\ar[r]^\id &\cdots \h^{2d-1}(X)_\C \ar[r]& \h^{2d-1}(X)_\C/F^{d} \cdots \\
	&V^{2d-1,k}_{n}(X) \ar[u] \ar[r]^\id &\cdots V^{2d-1,k}_{k+1}(X)\ar[u]\ar[r]&  \h^{2d-1,d}_{\rm dR}(X)\ar[u] \cdots   }
	\end{equation*}
	and $$F^{d+1}\h^{2d-1}(X)_\C\subset V^{2d-1,k}_{n}(X)=V^{2d-1,k}_{n-1}(X)=\cdots=V^{2d-1,k}_{k+1}(X)   $$
	Hence, according to Proposition \ref{prp:difflevels},  $\h_\sharp^{2d-1,d}(X)$ can be viewed as a formal Hodge structure of level $\le d$.
\end{exm}
\begin{prp}
	For any $n$ and $1\le p\le n$, the association $X\mapsto \h^{n,p}_{\sharp}(X)$ induces a contravariant functor from the category of proper complex algebraic schemes to  the category ${\sf FHS}_n$.
\end{prp} 
\begin{proof} It is enough to prove the claim for $p=n$.  We know that $\h^n(X):=\h^{n}(X_{\rm an},\Z)$ along with its mixed Hodge structures is functorial in $X$, so for any $f:X\to Y$ we have $\h^{n}(f):\h^{n}(Y)\to \h^{n}(X)$. Also by the theory of K\"ahler differentials there exist a map of complexes of sheaves over $X$, $\phi_{\bullet}:f^{*}\Omega_{Y}^{\bullet}\to \Omega_{X}^{\bullet}$, inducing
$$
\alpha:\h^{n}(X,f^{*}\Omega_{Y}^{<r})\longrightarrow\h^{n}(X,\Omega_{X}^{<r})
$$
Moreover there exists $\beta: \h^{n}(Y,\Omega_{Y}^{<r})\to \h^{n}(X,f^{*}\Omega_{Y}^{<r})$. For it is sufficient to construct a map $\beta':\h^{n}(Y,\Omega_{Y}^{<r})\to \h^{n}(X,f^{-1}\Omega_{Y}^{<r})$. So let $I^{\bullet}$ (resp. $J^{\bullet}$) an injective resolution\footnote{By injective resolution of  a complex of sheaves $A^{\bullet}$ we mean a quasi isomorphism $A^{\bullet}\to I^{\bullet}$, where $I^{\bullet}$ is a complex of injective objects.} of $\Omega_{Y}^{<r}$ (resp. $f\inv \Omega_{Y}^{<r}$). Using that $f\inv$ preserves quasi-isomorphisms,  we have the commutative diagram
\\
\makebox[\textwidth][c]{
\xymatrix{
 f^{-1}\Omega_{Y}^{<r}\ar[d]^{quis}  \ar[r]^{quis}&J^{\bullet}   \\  
f\inv I^{\bullet}\ar[ur]_{\exists \gamma}
}}
where the existence of $\gamma$ follows from the fact that $J^{\bullet}$ is injective. So
we have defined a map $\psi_{r}:\h^{n}(Y,\Omega^{<r})\to\h^{n}(X,\Omega^{<r})$. \\
Now choosing $I^{\bullet}_{r},J^{\bullet}_{r}$ for any $r$ it's easy to see that the maps $\psi_{r}$ fit in the commutative diagram
\\
\makebox[\textwidth][c]{
\xymatrix{
\cdots \ar[r] &  \h^{n}(Y,\Omega^{<r})  \ar[d]^{\psi_{r}}\ar[r]&   \h^{n}(Y,\Omega^{<r-1})  \ar[d]^{\psi_{r-1}}\ar[r]& \cdots \\
\cdots\ar[r]   &\h^{n}(X,\Omega^{<r})  \ar[r]&   \h^{n}(X,\Omega^{r-1}) \ar[r]& \cdots
}} 
Now it is straightforward to check that $\h^{n,n}_{\sharp}(g\circ f)=\h_\sharp^{n,n}(f)\circ \h^{n,n}_\sharp(g)$, for any $f:X\to Y$, $g:Y\to Z$.
\end{proof}
\begin{exm}[No K\"unneth]
	 Let $X,Y$ be complete, connected, complex varieties. Then by K\"unneth formula follows
	$$
	\h^{1}((X\x Y)_{\rm an},?)=\h^{1}(X_{\rm an},?)\oplus \h^{1}(Y_{\rm an},?)\qquad ?=\Z,\ \O
	$$
	so that $\h^{1}_{\sharp}(X\x Y)=\h^{1}_{\sharp}(X)\oplus \h^{1}_{\sharp}(Y)$. But as soon as we move in degree 2 there is no hope for a good formula. With the same notation we get
	$$
	\h^{2}((X\x Y))_\Q= \h^{2}(X)_\Q\oplus \h^{1}(X)_\Q\ox \h^{1}(Y)_\Q\oplus \h^{2}(Y)_\Q
	$$
	which is the usual decomposition of singular cohomology. Let $p:X\x Y\to X$, $q:X\x Y\to Y$ the two projections; note that 
	$$
	\O_{X\x Y}\to \Omega_{X\x Y}^{1}=\sigma^{<2}\left( p^{*}(\O_{X}\to \Omega^{1}_{X})\ox q^{*}(\O_{Y}\to \Omega^{1}_{Y})\right)
	$$
	hence there is a canonical map
	$$
	\h^{2}(X\x Y, p^{*}(\Omega_{X}^{<2})\ox q^{*}(\Omega^{<2}_{Y}))=\oplus_{i=0}^{2}\h^{2-i,2}_{\rm dR}(X)\ox \h^{i,2}_{\rm dR}(Y)\rightarrow \h^{2,2}_{\rm dR}(X\x Y)
	$$
	which is not necessarily an isomorphism. From this follows that we cannot have a K\"unneth formula for $\h_\sharp^{2,2}(X\x Y)$.
\end{exm}
\subsection{The generalized Albanese of Esnault-Srinivas-Viehweg} Let $X$ be a proper and irreducible algebraic scheme of dimension $d$ over $\C$. Then there exists an algebraic group, say ${\rm ESV}(X)$, such that ${\rm ESV}(X)_{\rm an}=\h^{2d-1}(X,\Omega^{<d})/\h^{2d-1}(X_{\rm an},\Z)$ and it fits in the following commutative diagram with exact rows
\begin{equation*}
\xymatrix{
0\ar[r] & \Ker c \ar[d]^\rho \ar[r]& \frac{ \h^{2d-1}(X)_\C }{ \h^{2d-1}(X) } \ar[d]^\alpha\ar[r]^c & J^d(\h^{2d-1}(X))\ar[d]^\id\ar[r]&0\\
0\ar[r]& \Ker \theta\ar[r]& \frac{\h^{2d-1,d}_{\rm dR}(X)}{\h^{2d-1}(X) }\ar[r]^\theta& J^d(\h^{2d-1}(X))\ar[r]&0}
\end{equation*}
where $\alpha $ is induced by de canonical map of complexes of analytic sheaves $\C\to \Omega^{<d}$. (See \cite[Theorem 1, Lemma 3.1]{esnault.srinivas.viehweg})

Recall that the formal Hodge structure (of level $\le 2d-1$) $\h^{2d-1,d}_\sharp(X)$ can be viewed as a fhs of level $\le d$ (see \ref{exm:albcoho}) represented by the following diagram
\begin{equation*}
\xymatrix{
\h^{2d-1}(X)\ar[dr]^h\ar[r]& \h^{2d-1}(X)_\C/F^d\ar[r]&\cdots  \h^{2d-1}(X)_\C/F^1\\
 & \h^{2d-1,d}_{\rm dR}(X)\ar[u]\ar[r]&\cdots \h^{2d-1,1}_{\rm dR}(X)\ar[u] \ .}
\end{equation*}
\begin{prp}
	\label{prp:ext-esv} There is an isomorphism of complex connected Lie groups (not only of abelian groups!)
	\[
		{\rm ESV}(X)_{\rm an}\iso \Ext^1_{\sf FHS_{d}}(\Z(-d),\h^{2d-1,d}_\sharp(X))
	\]
	where $\Z(-d)$ is the Tate structure of type $(d,d)$ viewed as an \'etale formal Hodge structure.
\end{prp}
\begin{proof} Step 1. By \cite{bv:fht} there is a canonical isomorphism of Lie groups
	\[
		{\rm ESV}_{\rm an}(X)\iso\Ext_{{}^t\mathcal{M}_1^{\rm a}}^1([\Z\to 0],[0\to {\rm ESV}(X)])\iso \Ext^1_{\sf FHS_1(1)}(\Z(0),T_{\oint}({\rm ESV}(X)))
	\]
	(recall that in \cite{bv:fht} ${\sf FHS}_1(1)$ is simply denote by ${\sf FHS}_1$; ${}^t\mathcal{M}_1^{\rm a}$ is the category of generalized 1-motives with torsion) where $T_{\oint}({\rm ESV}(X))$ is the formal Hodge structure represented by
	\begin{equation*}
	\xymatrix{
	\h^{2d-1}(X)(d)\ar[dr]\ar[r] &  	\h^{2d-1}(X)_\C(d)/F^0\\
	& \h^{2d-1,d}_{\rm dR}(X) \ar[u]}
	\end{equation*}

Step 2.	Up to a twist by $-d$ we can view $T_{\oint}({\rm ESV}(X))$  as an object of ${\sf FHS}_d$, say $(H_\et,V)$ with $H_\et=\h^{2d-1}(X)$, $V_d= \h^{2d-1,d}_{\rm dR}(X)$, $V_i=0$ for $1\le i< d$. It is easy to check that $\Ext^1_{{\sf FHS}_1(1)}(\Z(0),T_{\oint}({\rm ESV}(X)))=\Ext^1_{{\sf FHS}_d}(\Z(-d),(H_\et,V))$. Then applying $\Ext^1_{{\sf FHS}_d}(\Z(-d),-)$ to the canonical inclusion $(H_\et,V)\subset \h^{2d-1,d}_\sharp(X)$ we get a natural map
	\[
		 \Ext^1_{{\sf FHS}_1(1)}(\Z(0),T_{\oint}({\rm ESV}(X)))\rightarrow \Ext^1_{{\sf FHS}_{d}}(\Z(-d),\h^{2d-1,d}_\sharp(X))
	\]
which is an isomorphism by \eqref{eq:genext}.
\end{proof}
\subsection{The generalized Albanese of Faltings and W\"ustholz} Let $U$ be a smooth algebraic scheme over $\C$. Then it is possible to construct  a smooth compactification, i.e. $\exists$ $j:U\to X$ open embedding with $X$ proper and smooth. Moreover we can suppose that the complement $Y:=X\setminus U$ is a normal crossing divisor.\footnote{It is possible to replace $\C$ with a field $\k$ of characteristic zero. In that case we must assume that there exists a $\k$ rational point in order to have ${\rm FW}(Z)$ defined over $\k$.} 
\begin{rmk}
	\label{rmk:serrealb} There is a commutative diagram (See \cite[\S 3]{lekaus})
	\begin{equation*}
	\xymatrix{
	0\ar[r]&\h^0(X_{\rm an},\Omega^1(\log Y))\ar[d]^a\ar[r]& \h^1(U)_\C\ar[d]^\id\ar[r]&\h^{1,1}_{\rm dR}(X)\ar[d]^b \ar[r]&0\\
	0\ar[r]&\h^1(\Gamma(U_{an},\Omega^\bullet))\ar[r]& \h^1(U)_\C\ar[r]&\h^{1,1}_{\rm dR}(U)}
	\end{equation*}
	hence, by the snake lemma, $\Ker b\iso \Coker a$. We identify these two $\C$-vector spaces and we denote both by $K$. 

	For any $Z\subset K$ sub-vector space we define the $\C$-linear map $\alpha_Z:\h^1(X,\O)^*\to Z^*$ as the dual of the canonical inclusion $Z\subset \h^1(X,\O)$.	
\end{rmk}

\begin{dfn}[The generalized Albanese of Serre]
	 We know that 
	\[
		\h^1(U)(1)=T_{Hodge}([{\rm Div}_Y^0(X)\to \Pic^0(X)])
	\]
	and that the generalized Albanese of Serre is the Cartier dual of the above 1-motive, i.e.
	\[
		[0\to {\rm Ser}(U)]=[{\rm Div}_Y^0(X)\to \Pic^0(X)]\dual
	\]

	Note that by construction ${\rm Ser}(U)$ is a semi-abelian group scheme corresponding to the mixed Hodge structure $\h^1(U)(1)\dual:=\iHom_{\sf MHS}(\h^1(U)(1),\Z(1))$.

	 The universal vector extension of ${\rm Ser}(U)$ is
	\[
		0\to \underline{\omega}_{\Pic^0(X)}\to {\rm Ser}(U)^\natural\to {\rm Ser}(U)\to 0
	\]
	this follows by the construction of ${\rm Ser}(U)$ as the Cartier dual of $[{\rm Div}_Y^0(X)\to \Pic^0(X)]$ and \cite{bv.bertapelle:sharpderham} lemma 2.2.4. \\
	Recall that $\Lie(\Pic^0(X))= \h^1(X,\O)$, then $\underline{\omega}_{\Pic^0(X)}(\C)=\h^1(X,\O)^*$.
\end{dfn} 
\begin{dfn}[The gen. Albanese of Faltings and W\"ustholz]\label{FWconstruction}
	 We define an algebraic group ${\rm FW}(Z)$ (depending on $U$ and the choice of the vector space $Z$) to be the vector extension of ${\rm Ser}(U)$  by $Z^*$ defined by
	\[
	 \alpha_Z\in \Hom_\C(\h^1(X,\O)^*, Z^*)\iso\Hom_\C(\omega_{\Pic^0(X)},Z^*)\iso	\Ext^1({\rm Ser}(U),Z^*)
	\]
	i.e.  ${\rm FW}(Z)$ is the following push-forward
	\begin{equation*}
	\xymatrix{
	  0\ar[r]&\h^1(X,\O)^*\ar[d]^{\alpha_Z}\ar[r]&  {\rm Ser}(U)^\natural\ar[d]\ar[r]&{\rm Ser}(U)\ar[d]^\id\ar[r]&0\\
	0\ar[r]&Z^*\ar[r]& {\rm FW}(Z)\ar[r]& {\rm Ser}(U)\ar[r]&0}
	\end{equation*}
\end{dfn}
%
\begin{prp}
	 With the above notation consider the formal Hodge structure $(H_\et,V)\in \sf FHS_1$ represented by
	\begin{equation*}
	\xymatrix{
	\h^{1}(U)(1)\dual\ar[dr]^h\ar[r]& \h^{0}(X_{\rm an},\Omega^1(\log Y))^*\\
	& \h^{1}(\Gamma(U_{\rm an},\Omega^{\bullet}))^*\ar[u]^{a^*}}
	\end{equation*}
	(This diagram is the dual of the left square in remark \ref{rmk:serrealb}). Recall that $K=\Ker a$. Then
	\[
		{\rm FW}(K)_{\rm an}\iso \Ext^1_{{\sf FHS}_1}(\Z(-1),(H_\et,V))
	\]
\end{prp}
\begin{proof}
	It is a direct consequence of \ref{prp:genext}.
\end{proof}

\end{document}